\documentclass[11pt]{amsart}

\usepackage{amssymb,epsfig}
\usepackage{amsmath,verbatim,psfrag}

\usepackage[hmargin=1.5in, vmargin=1.5in]{geometry}

\newcommand{\nc}{\newcommand}
\nc{\nt}{\newtheorem}
\nc{\dmo}{\DeclareMathOperator}

\nt{theorem}{Theorem}
\nt{prop}[theorem]{Proposition}
\nt{lem}[theorem]{Lemma}
\nt{fact}[theorem]{Fact}
\nt{cor}[theorem]{Corollary}
\nt{conj}[theorem]{Conjecture}
\nt{question}[theorem]{Question}
\nt{problem}[theorem]{Problem}
\nt*{mainlemma}{Main Lemma}

\dmo{\Mod}{Mod}
\dmo{\Homeo}{Homeo}

\nc{\N}{\mathcal{N}}
\nc{\T}{\mathcal{T}}

\dmo{\Out}{Out}

\dmo{\Comm}{Comm}

\dmo{\GL}{GL}
\dmo{\SL}{SL}
\dmo{\PSL}{PSL}
\dmo{\PGL}{PGL}
\dmo{\SO}{SO}
\dmo{\Sp}{Sp}
\dmo{\PSp}{PSp}

\dmo{\vcd}{vcd}
\dmo{\rank}{rank}

\nc{\Z}{\mathbb{Z}}
\nc{\R}{\mathbb{R}}
\nc{\Q}{\mathbb{Q}}
\nc{\C}{\mathbb{C}}

\nc{\p}[1]{\medskip\paragraph{{\bf #1}}}
\nc{\margin}[1]{\marginpar{\scriptsize #1}}

\nc{\bl}{ \begin{list}{$\cdot$}{
\setlength{\leftmargin}{.5in}
\setlength{\rightmargin}{.5in}
\setlength{\parsep}{0.5ex plus .2ex minus 0ex}
\setlength{\itemsep}{0.2ex plus 0.2ex minus 0ex}
}
}

\nc{\el}{\end{list}}

\title[Abstract commensurators of Right-angled Artin groups]{Abstract commensurators of Right-angled Artin groups and Mapping class groups}

\begin{document}
	
\input{epsf.sty}

\author{Matt Clay}

\author{Christopher J. Leininger}

\author{Dan Margalit}

\address{Matt Clay \\ Department of Mathematics \\ 301 SCEN - 1 University of Arkansas \\  Fayetteville, AR 72701 \\ mattclay@uark.edu}

\address{Christopher J. Leininger\\ Dept. of Mathematics, University of Illinois at Urbana-Champaign \\ 273 Altgeld Hall, 1409 W. Green St. \\ Urbana, IL 61802\\ clein@math.uiuc.edu}

\address{Dan Margalit \\ School of Mathematics\\ Georgia Institute of Technology \\ 686 Cherry St. \\ Atlanta, GA 30332 \\  margalit@math.gatech.edu}

\thanks{The authors gratefully acknowledge support from the National Science Foundation.}



\begin{abstract}
We prove that, aside from the obvious exceptions, the mapping class group of a compact orientable surface is not abstractly commensurable with any right-angled Artin group. Our argument applies to various subgroups of the mapping class group---the subgroups generated by powers of Dehn twists and the terms of the Johnson filtration---and additionally to the outer automorphism group of a free group and to certain linear groups. 
\end{abstract}

\maketitle

There are many analogies and interconnections between the theories of right-angled Artin groups on one hand and mapping class groups on the other hand.  For instance, by the work of Crisp and Wiest \cite{crispwiest}, the work of Koberda \cite{koberda}, and the work of the first two authors with Mangahas \cite{clm}, there is an abundance of injective homomorphisms from right-angled Artin groups to mapping class groups.  Also, the last two authors proved \cite{pbfree} that any two elements of the pure braid group either generate a free group or a free abelian group---a property shared by all right-angled Artin groups \cite[Theorem 1.2]{baudisch}. We are thus led to ask to what extent mapping class groups are the same as right-angled Artin groups.  

It is straightforward to see that most mapping class groups are not isomorphic to right-angled Artin groups, for instance because right-angled Artin groups are torsion free.  On the other hand, mapping class groups have finite-index subgroups that are torsion free, and so this leaves open the possibility that mapping class groups are abstractly commensurable to right-angled Artin groups, that is, that they have isomorphic finite-index subgroups.   We prove that, aside from a small number of exceptions, this is not the case.  We also extend this result to several classes of groups related to mapping class groups.  We start by recalling some definitions.

\medskip

To a finite graph $\Gamma$, we can associate a right-angled Artin group: this is the group with one generator for each vertex of $\Gamma$, and one defining relator for each edge, namely, the commutator of the two generators corresponding to the endpoints.  

Let $S_{g,n}$ denote a closed, connected, orientable surface of genus $g$ with $n$ marked points.  The mapping class group $\Mod(S_{g,n})$ is the group of homotopy classes of orientation-preserving homeomorphisms of $S_{g,n}$ preserving the set of marked points.  

As discussed in Koberda's paper \cite[Theorem 1.5]{koberda}, no finite-index subgroup of $\Mod(S_{g,n})$ injects into a right-angled Artin group if $g \geq 2$ and $(g,n) \neq (2,0)$; see also \cite{kapovichleeb}.  In particular, such mapping class groups are not abstractly commensurable with right-angled Artin groups.  The last statement has a quick proof: for a right-angled Artin group, the virtual cohomological dimension is equal to the maximal rank of a free abelian subgroup, while for mapping class groups these numbers---which are invariant under passage to finite-index subgroups---are equal if and only if  $g=0$, $g=1$, or $(g,n)=(2,0)$; see \cite[Theorem 4.1]{harer} and \cite[Theorem A]{blm}.

As for the other mapping class groups, the first two authors proved with Mangahas that $\Mod(S_{2,0})$ is not abstractly commensurable with any right-angled Artin group \cite[Proposition 7.2]{clm} and explained how to apply their method to the case of $\Mod(S_{0,n})$ with $n \geq 6$.   Our first theorem extends these results to the remaining groups $\Mod(S_{g,n})$, while at the same time giving a uniform argument for all cases.

\begin{theorem}
\label{thm:main}
Let $g,n \geq 0$ and assume that $3g+n \geq 5$.  No right-angled Artin group is abstractly commensurable with $\Mod(S_{g,n})$.
\end{theorem}

To prove Theorem~\ref{thm:main}, we consider a third invariant of the abstract commensurability class of a group $G$, namely, the abstract commensurator $\Comm(G)$.  This is the group of all isomorphisms between finite-index subgroups of $G$, up to restriction; see \cite[Section 5]{ivanov}.  In what follows, we say that a group is virtually abelian if it has an abelian subgroup of finite index.  

\begin{mainlemma}
If $G$ is a group that is not virtually abelian and where $\Comm(G)$ does not contain $(\Z/2\Z)^k$ for arbitrarily large $k$, then $G$ is not abstractly commensurable with a right-angled Artin group.
\end{mainlemma}

\begin{proof}

It is enough to show that if $A$ is a non-abelian right-angled Artin group, then $\Comm(A)$ contains $(\Z/2\Z)^k$ for arbitrarily large $k$.  The proof of this fact consists of two observations.  

The first observation is that if the defining graph $\Gamma$ for a non-abelian right-angled Artin group $A$ has $k$ vertices, then $\Comm(A)$ contains a subgroup isomorphic to $(\Z/2\Z)^k$.  The generators are the abstract commensurators obtained by inverting the elements of $A$ corresponding to the vertices of $\Gamma$ (each is a nontrivial element of $\Comm(A)$ since every finite-index subgroup of $A$ contains some power of this generator and no power of a generator is equal to its inverse).

The second observation is that (as $A$ is not abelian) $A$ contains a finite-index subgroup $A'$ that is a right-angled Artin group whose defining graph has  arbitrarily many vertices.  Since $\Comm(A) \cong \Comm(A')$, the result follows.

To prove the second observation, we choose a vertex $v$ of $\Gamma$ for which the star of $v$ is not all of $\Gamma$.  Such a vertex exists since $A$ is not abelian.  Consider the homomorphism $A \to \Z/m\Z$ obtained by sending the generator of $A$ corresponding to $v$ to 1 and the generators corresponding to all other vertices to 0.  The kernel has finite index in $A$ and is a right-angled Artin group whose defining graph is obtained by taking $m$ copies of $\Gamma$ and gluing along the $m$ copies of the star of $v$; see \cite[Section 11]{bks} and \cite[Corollary 5]{kk}.  Since the star of $v$ is not all of $\Gamma$, the new graph has more than $m$ vertices.  
\end{proof}

\medskip

We can deduce Theorem~\ref{thm:main} directly from the Main Lemma, as follows.  First,  $\Mod(S_{g,n})$ is not virtually abelian, for instance because nonzero powers of Dehn twists about two curves with nonzero geometric intersection fail to commute (and we can choose these powers so that they lie in any given subgroup of finite index) \cite[Section 3.3]{primer}. Also, for most of the surfaces covered by Theorem~\ref{thm:main}, Ivanov \cite[Theorem 5]{ivanov} and Korkmaz \cite[Theorem 3]{korkmaz} showed that $\Comm(\Mod(S_{g,n}))$ is isomorphic to the extended mapping class group $\Mod^{\pm}(S_{g,n})$, the group of homotopy classes of all (not-necessarily-orientation-preserving) homeomorphisms of $S_{g,n}$.  There are two exceptions: $\Comm(\Mod(S_{1,2})) \cong \Mod^\pm(S_{0,5})$ and $\Comm(\Mod(S_{2,0})) \cong \Mod^\pm(S_{0,6})$; see \cite[Proposition 7]{bm} and \cite[Theorem 1.2]{irmak}.  It follows from Kerckhoff's solution to the Nielsen realization problem \cite{kerckhoff} that $\Mod^\pm(S_{g,n})$ does not contain finite subgroups of arbitrary large cardinality; see, e.g., \cite[Section 7.2]{primer}.  Thus, $\Mod(S_{g,n})$ satisfies both hypotheses of the Main Lemma and Theorem~\ref{thm:main} follows immediately.  

\medskip

Our proof of the Main Lemma can be combined with a theorem of Bartholdi and Bogopolski \cite[Theorem 2.8]{bb} to prove that $\Comm(A)$ is not finitely generated, thus giving a different (but similar) proof that $\Comm(A)$ is not isomorphic to any  $\Mod^{\pm}(S)$.

\medskip

The assumptions in Theorem~\ref{thm:main} are in fact necessary, as $\Mod(S_{1,0}) \cong \Mod(S_{1,1}) \cong \SL_2(\Z)$ and $\Mod(S_{0,4}) \cong \PSL_2(\Z) \ltimes (\Z/2\Z \times \Z/2\Z)$ are commensurable with the free group $F_2$ and for $n \leq 3$ the group $\Mod(S_{0,n})$ is finite, hence abstractly commensurable with the trivial right-angled Artin group.

\medskip

Finally, the $(\Z/2\Z)^k$ subgroups of $\Comm(A)$ we construct further embed into the quasi-isometry group of $A$, and so using the theorem of Behrstock, Kleiner, Minsky, and Mosher that the quasi-isometry group of $\Mod(S_{g,n})$ is again an extended mapping class group \cite[Theorem 1.1]{bkmm}, we can conclude the following strengthening of Theorem~\ref{thm:main}.

\begin{theorem}
\label{thm:qi}
Let $g,n \geq 0$ and assume that $3g+n \geq 5$.  No right-angled Artin group is quasi-isometric to $\Mod(S_{g,n})$.
\end{theorem}

\p{Mapping class groups of surfaces with boundary} Let $S_{g,n}^b$ denote the surface obtained from $S_{g,n}$ by removing the interiors of $b$ disks, disjoint from each other and the marked points; we denote $S_{g,0}^b$ by $S_g^b$ and $S_{g,0}^0$ by $S_g$.  The mapping class group $\Mod(S_{g,n}^b)$ is the group of homotopy classes of orientation-preserving homeomorphisms of $S_{g,n}^b$ that restrict to the identity on the boundary.

\begin{theorem}
\label{thm:boundary}
Let $g,n,b \geq 0$ and assume that $3g+n+b \geq 5$.  No right-angled Artin group is abstractly commensurable with $\Mod(S_{g,n}^b)$.
\end{theorem}

As the braid group on $n$ strands is isomorphic to $\Mod(S_{0,n}^1)$, Theorem~\ref{thm:boundary} in particular implies that the braid group (or pure braid group) on $n \geq 4$ strands is not abstractly commensurable with any right-angled Artin group.

When $b > 0$, the virtual cohomological dimension of $\Mod(S_{g,n}^b)$ and the maximal rank of an abelian subgroup of $\Mod(S_{g,n}^b)$ are equal if and only if $g \in\{0,1\}$ \cite[Theorem 4.1]{harer}, and so to prove Theorem~\ref{thm:boundary} it suffices to consider these two values of $g$.  A fact special to these two cases is that $\Mod(S_{g,n}^b)$ is abstractly commensurable with the direct product $\Mod(S_{g,n+b}) \times \Z^b$; see Theorem~\ref{thm:virtual splitting} below.  Theorem~\ref{thm:boundary} thus follows from the next lemma and Theorem~\ref{thm:main}.

\begin{lem}
If a group $G$ is abstractly commensurable with $G' \times \Z^n$ for some group $G$ and some $n \geq 0$, then $G$ is abstractly commensurable with a right-angled Artin group if and only if $G'$ is abstractly commensurable with a right-angled Artin group.
\end{lem}

\begin{proof}

If a right-angled Artin group $A$ is abstractly commensurable with $G$, it is then abstractly commensurable with $G' \times \Z^n$.  From this it follows that $A$ is isomorphic to $A' \times \Z^n$ for some right-angled Artin group $A'$ (use \cite[Section III]{servatius}) and $A'$ is abstractly commensurable with $G'$ (cf. \cite[Lemma 2.2]{commbn}).  The other direction is trivial since a direct product of right-angled Artin groups is again a right-angled Artin group.  
\end{proof}


The genus zero and one cases of Theorem~\ref{thm:boundary} can also be proven by directly applying the Main Lemma, as in the proof of Theorem~\ref{thm:main}.  The group $\Comm(\Mod(S_{0,n}^1))$ was computed by the last two authors \cite{commbn} and a similar argument using Theorem~\ref{thm:virtual splitting} below and Section 3 of \cite{commbn} shows that 
\[ \Comm(\Mod(S_{g,n}^b)) \cong \Comm(\Mod(S_{g,n+b})) \ltimes(\GL_b(\Q) \ltimes (\Q^b)^\infty) \]
for $g\in\{0,1\}$ and $3g+n+b \geq 5$.   Again, these groups do not contain finite subgroups of arbitrarily large cardinality.


\p{Subgroups of the mapping class group} We will apply our Main Lemma to show that several other classes of groups are not abstractly commensurable with right-angled Artin groups: first for certain subgroups of the mapping class group, and then for the outer automorphism group of a free group and certain linear groups.

Let $\T_d(S_g)$ denote the subgroup of $\Mod(S_g)$ generated by the $d$th powers of all Dehn twists.  This group has infinite index in $\Mod(S_g)$ for $d \geq 11$ and $g \geq 2$ \cite{funar}.  It has been conjectured that $\T_d(S_g)$ is a right-angled Artin group for $d$ large \cite{funar}.  However, Ivanov's proof that $\Comm(\Mod(S_g)) \cong \Mod^\pm(S_g)$ (see \cite[Theorem 5]{ivanov}) carries over to show that $\Comm(\T_d(S_g)) \cong \Mod^\pm(S_g)$ for $g \geq 3$; see \cite[Corollary 2]{aramayonasouto} for an alternate argument.  As $\T_d(S_g)$ is not virtually abelian, we conclude the following.

\begin{theorem}
Let $g \geq 3$ and $d \geq 1$.  The group $\T_d(S_g)$ is not abstractly commensurable with any right-angled Artin group.
\end{theorem}

Let $\pi$ denote $\pi_1(S_g)$, and let $\pi^k$ denote the $k$th term of its lower central series: $\pi^1 = \pi$ and $\pi^{k+1} = [\pi,\pi^k]$.  The Johnson filtration of $\Mod(S_g)$ is the nested  sequence of groups $(\N_k(S_g))$ where $\N_k(S_g)$ is the kernel of the natural homomorphism $\Mod(S_g) \to \Out(\pi/\pi^{k+1})$.  The intersection of the $\N_k(S_g)$ is trivial.   The groups $\N_1(S_g)$ and $\N_2(S_g)$ are also known as the Torelli group and Johnson kernel of $S_g$.   For $g$ large enough, the abstract commensurators of these groups are all known to be isomorphic to $\Mod^\pm(S_g)$ \cite[Theorem 7]{farbivanov} \cite[Main Theorem 1]{commkg} \cite{bps}.  

\begin{theorem}
Let $k \geq 1$.  Let $g \geq 3$ if $k \leq 3$ and let $g \geq 4$ if $k > 3$.  No right-angled Artin group is abstractly commensurable with $\N_k(S_g)$.
\end{theorem}

\p{Outer automorphism groups of right-angled Artin groups} For $g \geq 1$, the group $\Mod^\pm(S_g)$ is isomorphic to $\Out(\pi_1(S_g))$; this is the Dehn--Nielsen--Baer theorem  \cite[Theorem 8.1]{primer}.  As $\pi_1(S_1) \cong \Z^2$ and as $\pi_1(S_g^b)$ is a free group for $b > 0$, we can obtain analogs of the mapping class group by considering groups of the form $\Out(A)$ where $A$ is a right-angled Artin group.  By results of Davis--Januszkiewicz \cite{dj} and Taylor \cite{taylor}, there are many embeddings of right-angled Artin groups into $\Out(A)$ where $A$ is $\Z^n$ or $F_n$, just like in the mapping class group case. 

By the combined work of Borel \cite{borel} and Margulis \cite{margulis}, the abstract commensurator of $\GL_n(\Z)\cong\Out(\Z^n)$ is $\PGL_n(\Q) \rtimes \Z/2\Z$ for $n \geq 3$; see e.g.~\cite[Section 7.3]{studenmund} for an exposition of these ideas.  Since $\PGL_n(\Q)$ does not contain arbitrarily large subgroups of the form $(\Z/2\Z)^k$ and since $\GL_n(\Z)$ is not virtually abelian, we obtain the following further consequence of our Main Lemma. 

\begin{theorem}
\label{thm:gl}
Let $n \geq 3$.  No right-angled Artin group is abstractly commensurable with $\Out(\Z^n) \cong \GL_n(\Z)$.
\end{theorem}

Wortman \cite{wortman} has pointed out that our Main Lemma also applies with $\GL_n(\Z)$ replaced by any lattice in a semisimple Lie group not locally isomorphic to $\SL_2(\R)$.  Indeed, for any such lattice $G$, Mostow--Prasad--Margulis superrigidity implies that $\Comm(G)$ is a subgroup of some $\GL_n(\C)$.  From the theory of Jordan canonical forms (in particular the fact that commuting matrices can be simultaneously put into normal form), we know that $\GL_n(\C)$ contains subgroups isomorphic to $(\Z/2\Z)^k$ only for $k \leq n$.  In particular, we can deduce the theorem of Koberda that for $k \geq 3$ no right-angled Artin group is abstractly commensurable with a lattice in $\SO(k,1)$ \cite[Theorem 1.14]{koberda}.  Additionally, Studenmund \cite{studenmund} has shown that the abstract commensurators of many lattices with nontrivial solvable radicals (e.g.~$\SL_n(\Z) \ltimes \Z^n$) are linear, so our argument applies to these lattices as well. 

\medskip

Farb and Handel proved that $\Comm(\Out(F_n))$ is isomorphic to $\Out(F_n)$ when $n \geq 4$ (it is not known whether $\Comm(\Out(F_3)) \cong \Out(F_3)$ or not) \cite{farbhandel}.  A finite subgroup of $\Out(F_n)$ can be identified with the symmetries of some fixed metric graph of rank $n$ \cite[Theorem 2.1]{culler} \cite{zimmermann}, and so there is a bound on the cardinality of such subgroups that depends only on $n$.  Since $\Out(F_n)$ is not virtually abelian, our Main Lemma also implies the following.

\begin{theorem}
\label{thm:out}
Let $n \geq 4$.   No right-angled Artin group is abstractly commensurable with $\Out(F_n)$.
\end{theorem}

Theorems~\ref{thm:gl} and~\ref{thm:out} tell us that (in most cases) the outer automorphism groups of the right-angled Artin groups $\Z^n$ and $F_n$ are not commensurable with right-angled Artin groups.  We also know that $\Out(\Z^2) \cong \Out(F_2) \cong \GL_2(\Z)$ is commensurable with $F_2$, and that the outer automorphism group of the right-angled Artin group associated to the linear graph with four vertices is $((\Z/2\Z) \ltimes (\Z/2\Z)^4) \ltimes \Z^4$ (cf. \cite[Proposition 2.15]{day}), which is abstractly commensurable with $\Z^4$.  We are therefore led to the following question.

\begin{question}
For which right-angled Artin groups $A$ is $\Out(A)$ abstractly commensurable with a right-angled Artin group?
\end{question}


\appendix

\setcounter{secnumdepth}{0}

\section{Appendix. Virtual splitting for low-genus mapping class groups}

In this appendix we prove a fact used in our proof of Theorem~\ref{thm:boundary}.  For the statement, we define the level two subgroup $\Mod(S_{g,n}^b)[2]$ of $\Mod(S_{g,n}^b)$ as the finite-index subgroup consisting of all elements that act trivially on the mod two homology of the surface obtained from $S_{g,n}^b$ by removing the marked points.  Note that when $g=0$ the group $\Mod(S_{0,n}^b)[2]$ is the same as the pure mapping class group; in particular $\Mod(S_{0,n}^1)[2]$ is isomorphic to the pure braid group on $n$ strands.

\begin{theorem}
\label{thm:virtual splitting}
Let $g \in \{0,1\}$, and assume $3g+n+b \geq 3$.  We have
\[ \Mod(S_{g,n}^b)[2] \cong \Mod(S_{g,n+b})[2] \times \Z^b. \]
\end{theorem}

\begin{proof}

It suffices to show that $\Mod(S_{g,n}^b)[2]$ splits as a direct product over its center, the free abelian group generated by the Dehn twists about the $b$ boundary components of $S_{g,n}^b$.  The theorem then follows from the fact that the natural inclusion $S_{g,n}^b \to S_{g,n+b}$ induces a well-defined surjective map of level two mapping class groups, and that the kernel of this map is the center \cite[Proposition 3.19]{primer}.

First we deal with the case $g=1$.  The group $\Mod(S_1^1)[2]$ is isomorphic to the pure braid group on three strands (combine \cite[Section 9.4.1]{primer} with \cite{arnold}).  The latter splits over its infinite cyclic center \cite[Section 9.3]{primer}, and so we are done in this case.  Denote by $s$ a retraction $\Mod(S_1^1)[2] \to Z(\Mod(S_1^1)[2])$ that defines the splitting.

Let $f_i : \Mod(S_{1,n}^b)[2] \to \Mod(S_1^1)[2]$ be the homomorphism obtained by forgetting the $n$ marked points and by capping each boundary component of $S_{1,n}^b$---except for the $i$th---with a disk.  Also, let $h_i : Z(\Mod(S_1^1)[2]) \to \Mod(S_{1,n}^b)[2]$ be the homomorphism that maps the Dehn twist about the boundary of $S_1^1$ to the Dehn twist about the $i$th boundary component of $S_{1,n}^b$.  The product of the maps $h_i \circ s \circ f_i$ is the desired retraction $\Mod(S_{1,n}^b)[2] \to Z(\Mod(S_{1,n}^b)[2])$.

The genus zero version is nearly identical.  The role of $\Mod(S_1^1)[2]$ is played by $\Mod(S_{0,2}^1)[2] \cong \Z$.  The only essential difference in this case is that there are many choices of homomorphism $\Mod(S_{0,n}^b)[2] \to \Mod(S_{0,2}^1)[2]$ available for each coordinate of the splitting.  
\end{proof}

\medskip

In contrast to Theorem~\ref{thm:virtual splitting}, it is known that $\Mod(S_g^1)$ does not have a finite-index subgroup that splits over its center when $g \geq 2$; see \cite[Proposition 5.10]{primer}.

\p{Acknowledgments} We would like to thank Jason Behrstock, Matt Day, Benson Farb, Thomas Koberda, Daniel Studenmund, Richard Wade, Kevin Wortman, and the referee for helpful comments and conversations.


\bibliographystyle{plain}
\bibliography{raagvsmcg}

\end{document}